\documentclass[oneside,english]{amsart}
\usepackage{amsthm}
\usepackage{amssymb}

\numberwithin{equation}{section} 
\numberwithin{figure}{section} 
\theoremstyle{plain}
\theoremstyle{plain}
\newtheorem{thm}{Theorem}
  \theoremstyle{plain}
  \newtheorem{cor}[thm]{Corollary}
  \theoremstyle{plain}
  \newtheorem{conjecture}[thm]{Conjecture}
  \theoremstyle{plain}
  \newtheorem{lem}[thm]{Lemma}
  \theoremstyle{definition}
  \newtheorem{defn}[thm]{Definition}
  \theoremstyle{remark}
  \newtheorem{claim}[thm]{Claim}

\numberwithin{thm}{section}

\usepackage{babel}

\begin{document}

\title{A dense G-delta set of Riemannian metrics without the finite blocking
property}

\author{Marlies Gerber and Wah-Kwan Ku}
\begin{abstract}
A pair of points $(x,y)$ in a Riemannian manifold $(M,g)$ is said
to have the finite blocking property if there is a finite set $P\subset M\setminus\{x,y\}$
such that every geodesic segment from $x$ to $y$ passes through
a point of $P$. We show that for every closed $C^{\infty}$ manifold
$M$ of dimension at least two and every pair $(x,y)\in M\times M$,
there exists a dense $G_{\delta}$ set, $\mathcal{G}$, of $C^{\infty}$
Riemannian metrics on $M$ such that $(x,y)$ fails to have the finite
blocking property for every $g\in\mathcal{G}$.
\end{abstract}

\email{gerber@indiana.edu, wku@indiana.edu}

\address{Department of Mathematics, Indiana University, Bloomington, IN 47405,
USA}

\maketitle

\section{Introduction}

\markboth{}{}Let $M$ be a closed $C^{\infty}$ manifold, and let
$g$ be a $C^{\infty}$ Riemannian metric on $M.$ We consider a geodesic
as a mapping $\gamma:I\to M$, where $I$ is an interval of positive
length, and $\gamma$ is parametrized by arc length. Two geodesics
$\gamma_{i}:I_{i}\to M$, $i=1,2$ will be considered to be the same
if and only if $\gamma_{1}=\gamma_{2}\circ\varphi$, where $\varphi$
is a translation that maps $I_{1}$ onto $I_{2}.$ Let $x$ and $y$
be points in $M$, possibly with $x=y$. When we say that a geodesic
$\gamma:[c,d]\to M$ is from $x$ to $y$, we mean $\gamma(c)=x$
and $\gamma(d)=y$.

Given a Riemannian metric $g$ on $M$, a ${\it blocking\ set}$ for
$(x,y)$ is defined to be a subset $P$ of $M\setminus\{x,y\}$ such
that every geodesic from $x$ to $y$ passes through a point in $P$.
The pair $(x,y)\in M\times M$ is said to have the \emph{finite blocking
property} for $g$ if there exists a finite blocking set for $(x,y)$.
If every $(x,y)\in M\times M$ has the finite blocking property, then
$(M,g)$ is called \emph{secure}. (See \cite{bib4} and \cite{bib5}
for an explanation of this terminology.) A Riemannian manifold $(M,g)$
is called \emph{insecure} if it is not secure, and it is called \emph{totally
insecure} if no pair $(x,y)$ has the finite blocking property. Furthermore,
it is called \emph{uniformly secure} if there exists a positive integer
$n$ such that any pair of points $(x,y)$ has a blocking set with
at most $n$ elements.

Given a manifold $M,$ it is natural to ask the following:

\noindent \textbf{Question.} \emph{Which pairs of points $(x,y)\in M\times M$
and which Riemannian metrics $g$ on $M$ are such that $(x,y)$ has
the finite blocking property for $g$?}

Our contribution in this direction is Theorem \ref{thm:main} below,
which says that any given pair of points $(x,y)$ fails to have the
finite blocking property for a dense $G_{\delta}$ set of metrics.
We will give the proof in Section 3.

We let $\mathbb{G}$ denote the set of $C^{\infty}$ Riemannian metrics
on $M$. For $k=1,2,\dots,\infty$, there exists a complete metric
on $\mathbb{G}$ whose topology coincides with the $C^{k}$ topology
on $\mathbb{G}$. In particular, the Baire category theorem applies
to $\mathbb{G}$ with the $C^{k}$ topology. When we refer to the
$C^{k}$ topology on $M\times\mathbb{G}$ or $M\times M\times\mathbb{G}$,
we mean the product topology, where we take the manifold topology
on $M$ and the $C^{k}$ topology on $\mathbb{G}$.
\begin{thm}
\label{thm:main} Let $M$ be a closed $C^{\infty}$ manifold of dimension
at least two, and let $\mathbb{G}$ be the space of $C^{\infty}$
Riemannian metrics on $M$. The following three statements hold.
\begin{enumerate}
\item \label{enu:thm1-part1}Let $x$ and $y$ be two points in $M$, possibly
with $x=y$. Let $\mathcal{G}=\{g\in\mathbb{G}:(x,y){\rm \ fails\ to\ have\ the\ finite\ blocking\ property\ for\ }g\}$.
Then $\mathcal{G}$ contains the intersection of a countable collection
of sets that are $C^{1}$-open and $C^{\infty}$-dense in $\mathbb{G}$.
Thus, $\mathcal{G}$ contains a dense $G_{\delta}$ set in the $C^{k}$
topology for $k=1,2,\dots,\infty.$
\item \label{enu:thm1-part2} Let $\tilde{\mathcal{G}}=\{(x,y,g)\in M\times M\times\mathbb{G}:(x,y){\rm \ fails\ to\ have\ the\ finite\ blocking}$
${\rm property\ for\ }g\}$. Then $\tilde{\mathcal{G}}$ contains
the intersection of a countable collection of sets that are $C^{1}$-open
and $C^{\infty}$-dense in $M\times M\times\mathbb{G}$. Thus, $\tilde{\mathcal{G}}$
contains a dense $G_{\delta}$ set in the $C^{k}$ topology for $k=1,2,\dots,\infty.$
\item \label{enu:thm1-part3} Let $\hat{\mathcal{G}}=\{(x,g)\in M\times\mathbb{G}:(x,x){\rm \ fails\ to\ have\ the\ finite\ blocking\ property}$
${\rm for\ }g\}$. Then $\hat{\mathcal{G}}$ contains the intersection
of a countable collection of sets that are $C^{1}$-open and $C^{\infty}$-dense
in $M\times\mathbb{G}$. Thus, $\hat{\mathcal{G}}$ contains a dense
$G_{\delta}$ set in the $C^{k}$ topology for $k=1,2,\dots,\infty.$
\end{enumerate}
\end{thm}
From (\ref{enu:thm1-part2}) and (\ref{enu:thm1-part3}), we can deduce
the following corollary.
\begin{cor}
Let $M$ be a closed $C^{\infty}$ manifold of dimension at least
two and suppose $k\in\{1,2,\dots,\infty\}.$
\begin{enumerate}
\item \label{enu:cor1-part1}There exists a dense $G_{\delta}$ set $\mathcal{G}_{1}$
in $\mathbb{G}$ with the $C^{k}$ topology, so that for each $g\in\mathcal{G}_{1}$,
there is a dense $G_{\delta}$ subset $\mathcal{R}_{1}:=\mathcal{R}_{1}(g)$
of $M\times M$ such that each $(x,y)\in\mathcal{R}_{1}$ fails to
have the finite blocking property for $g$.
\item \label{enu:cor1-part2}There exists a dense $G_{\delta}$ set $\mathcal{G}_{2}$
in $\mathbb{G}$ with the $C^{k}$ topology, so that for each $g\in\mathcal{G}_{2}$,
there is a dense $G_{\delta}$ subset $\mathcal{R}_{2}:=\mathcal{R}_{2}(g)\subseteq M$
such that for each $x\in\mathcal{R}_{2}$, $(x,x)$ fails to have
the finite blocking property for $g$.
\end{enumerate}
\end{cor}
V. Bangert and E. Gutkin obtained stronger results for the case when
the dimension of $M$ is two and the genus is positive \cite{bib12}.
They proved that if $M$ has genus greater than one, then every Riemannian
metric is totally insecure. Moreover, if $M$ has genus one, they
showed that non-flat metrics are insecure and a $C^{2}$-open, $C^{\infty}$-dense
set of metrics are totally insecure. These results provide evidence
that (c) follows from (a) in the following conjecture, which originally
appeared in \cite{bib5} and \cite{bib8}. A proof that (c) implies
(b) is given in \cite{bib6}.
\begin{conjecture}
\label{conjecture} Let $(M,g)$ be a closed $C^{\infty}$ Riemannian
manifold. The following statements are equivalent.
\begin{enumerate}
\item [(a)]$(M,g)$ is secure.
\item [(b)]$(M,g)$ is uniformly secure.
\item [(c)]$g$ is a flat metric.
\end{enumerate}
\end{conjecture}
While Conjecture \ref{conjecture} concerns the finite blocking property
for \emph{all pairs of points}, Theorem \ref{thm:main} shows that
the finite blocking property can be destroyed for \emph{any given
pair of points}, under some small perturbation of metric.

In the next section, we will present some results which will be used
to prove Theorem \ref{thm:main}. We refer the reader to \cite{bib3}
for background information about geodesics and conjugate points.

We thank  Chris Connell for a helpful conversation that led to an
improvement to our original version of Theorem \ref{thm:main}.

\section{Some preliminary results}

We begin with the following classical result by J. P. Serre \cite{bib1},
\cite{bib13}, \cite{bib14}, \cite{bib11}.
\begin{thm}
\label{Serre} Let $(M,g)$ be a closed $C^{\infty}$-Riemannian manifold,
and let $x,y\in M$. Then there exist infinitely many geodesics from
$x$ to $y$.
\end{thm}
The following lemma allows us to {}``merge'' two foliations by geodesics
for a Riemannian metric $g$ into a new foliation by geodesics for
a small perturbation of $g$, provided the two original foliations
are $C^{\infty}$-close.

For $a,b>0$, we let $I_{a}$ denote the open interval $(-a,a)\subset\mathbb{R}$,
and we let $B_{b}$ denote the open ball $\{\mathbf{w}\in\mathbb{R}^{n-1}:|\mathbf{w}|<b\}$,
where $n$ is the dimension of the manifold $M$ under consideration.
\begin{lem}
\label{foliation merging} Let $(M,g)$ be a closed $C^{\infty}$
Riemannian manifold of dimension $n\ge2$, and let $\mathbb{G}$ be
the set of $C^{\infty}$ Riemannian metrics on $M$. Suppose $\mathcal{N}$
is an open neighborhood of $g$ in $\mathbb{G}$ with the $C^{\infty}$
topology. Choose $a,b>0$, and let $\mathcal{F}=\{f\in C^{\infty}(I_{a}\times B_{b},M)$
$|$ $f$ satisfies \emph{(i), (ii), (iii)} below $\}$.
\begin{enumerate}
\item [(i)]The map $f$ is a $C^{\infty}$-diffeomorphism onto its image.
\item [(ii)]For all $\mathbf{p}\in B_{b}$ , the map $t\mapsto f(t,\mathbf{p})$
, for $t\in I_{a}$, is a geodesic (for the metric $g$).
\item [(iii)]For all $t\in I_{a}$, the $(n-1)$-dimensional submanifold
$\{f(t,\mathbf{p}):\mathbf{p}\in B_{b}\}$ is perpendicular (in the
metric $g$) to all the geodesics in ${\rm (ii)}$.
\end{enumerate}
We consider $\mathcal{F}$ with the relative topology induced from
the $C^{\infty}$compact-open topology on $C^{\infty}(I_{a}\times B_{b},M)$.
Suppose $f_{0}\in\mathcal{F}$. Then there exists an open neighborhood
$\mathcal{F}_{0}\subseteq\mathcal{F}$ of $f$ such that for all $f_{1},f_{2}\in\mathcal{F}_{0}$,
there exists $\tilde{g}\in\mathcal{N}$ such that the following conditions
are satisfied.
\begin{enumerate}
\item $\tilde{g}$ agrees with $g$ on the complement of $f_{1}(I_{a/2}\times B_{b/2})\cap f_{2}(I_{a/2}\times B_{b/2})$.
\item There is a family of $\tilde{g}$-geodesics $\gamma_{\mathbf{p}}:I_{a}\rightarrow f_{1}(I_{a}\times B_{b})\cup f_{2}(I_{a}\times B_{b})$,
for $\mathbf{p}\in B_{b/4}$, such that \[
\gamma_{\mathbf{p}}(t)=\begin{cases}
f_{1}(t,\mathbf{p}), & \text{if $t\in(-a,-a/4)$;}\\
f_{2}(t,\mathbf{p}), & \text{if }t\in(a/4,a).\end{cases}\]

\item If $f_{1}(t,\mathbf{0})=f_{2}(t,\mathbf{0})$ for all $t\in I_{a}$,
then $\gamma_{\mathbf{0}}(t)=f_{1}(t,\mathbf{0})$. This implies that
the map $t\mapsto f_{1}(t,\mathbf{0})$ for $t\in I_{a}$, is a geodesic
for $\tilde{g}$ as well.
\end{enumerate}
\end{lem}
\begin{proof}
Let $(a_{i})_{0\le i\le5}$ and $(b_{j})_{0\le j\le5}$ be strictly
decreasing sequences of positive numbers, where $a_{0}=a,a_{3}=a/2,a_{5}=a/4,b_{0}=b,b_{1}=b/2$,
and $b_{5}=b/4$. Let $R_{i,j}=I_{a_{i}}\times B_{b_{j}}$, for $0\le i,j\le5$.

Let $h:\mathbb{R}\rightarrow[0,1]$ be a $C^{\infty}$ function such
that \[
h(t)=\begin{cases}
0, & \text{if $t\leq-a_{5}$};\\
1, & \text{if $t\geq a_{5}$},\end{cases}\]
 and let $H:M\rightarrow[0,1]$ be a $C^{\infty}$ function such that
\[
H(x)=\begin{cases}
0, & \text{if $x\in M\setminus f_{0}(R_{3,3})$};\\
1, & \text{if $x\in f_{0}(\overline{R_{4,4}})$}.\end{cases}\]
 Given $f_{0}\in\mathcal{F}$, the required open neighborhood $\mathcal{F}_{0}$
will be chosen so that functions $f_{1,}f_{2}\in\mathcal{F}_{0}$
satisfy the properties given below. We begin by requiring $f_{1},f_{2}$
to be sufficiently close to $f_{0}$ in the $C^{0}$ topology so that

\begin{equation}
f_{2}(\overline{R_{i+1,j+1}})\subseteq f_{1}(R_{i,j})\text{ and }f_{1}(\overline{R_{i+1,j+1}})\subseteq f_{2}(R_{i,j}),\text{ for }0\le i,j\le4.\label{eq:rectangle-containment}\end{equation}

We define $\phi:\overline{R_{1,1}}\rightarrow R_{0,0}$ by \[
\phi(t,\mathbf{p})=(1-h(t))(t,\mathbf{p})+h(t)(f_{1}^{-1}\circ f_{2}(t,\mathbf{p}))\]
 for $(t,\mathbf{p})\in\overline{R_{1,1}}$, where `$+$' denotes
the usual vector addition in $\mathbb{R}^{n}$. We have $\phi(t,\mathbf{p})\in R_{0,0}$
, because $f_{1}^{-1}\circ f_{2}(\overline{R_{1,1}})\subseteq R_{0,0}$
(by $(\ref{eq:rectangle-containment})$), and $\phi(t,\mathbf{p})$
is a convex combination of $f_{1}^{-1}\circ f_{2}(t,\mathbf{p})$
and $(t,\mathbf{p})$.

Next we consider $\hat{f}:=f_{1}\circ\phi:\overline{R_{1,1}}\rightarrow f_{1}(R_{0,0})$.
If $f_{1}$ and $f_{2}$ are close to $f_{0}$ in $C^{\infty}(R_{0,0},M)$,
then $\phi$ is close to the inclusion map $\overline{R_{1,1}}\hookrightarrow R_{0,0}$
in $C^{\infty}(\overline{R_{1,1}},R_{0,0})$, and $\hat{f}$ is close
to $f_{0}$ in $C^{\infty}(\overline{R_{1,1}},M)$. We require $f_{1}$
and $f_{2}$ to be sufficiently close to $f_{0}$ in $C^{\infty}(R_{0,0},M)$
so that the following four conditions are satisfied: \begin{equation}
\hat{f}:\overline{R_{1,1}}\rightarrow M{\rm \text{ is a diffeomorphism onto its image,}}\label{eq:f-hat-is-diffeo}\end{equation}
\begin{equation}
\hat{f}(\overline{R_{i+1,j+1}})\subseteq f_{0}(R_{i,j})\cap f_{1}(R_{i,j})\cap f_{2}(R_{i,j})\text{, for }0\le i,j\le4,\label{eq:f-hat-rectangle}\end{equation}
\begin{equation}
f_{0}(\overline{R_{3,3}})\subseteq\hat{f}(R_{2,2}),\text{ and}\label{eq:f1_fhat}\end{equation}
\begin{equation}
(f_{1}((-a_{1},-a_{2}]\times B_{b_{2}})\cup f_{2}([a_{2},a_{1})\times B_{b_{2}}))\cap\hat{f}(\overline{R_{5,2}})=\emptyset.\label{eq:f1_f2_fhat}\end{equation}

For $(t,\mathbf{p})=(t,p_{1},\dots,p_{n-1})\in R_{2,2}$, we define
a Riemannian metric $\hat{g}$ at $\hat{f}(t,\mathbf{p})\in\hat{f}(R_{2,2})$
by\begin{align}
\hat{g}\left(\frac{\partial\hat{f}}{\partial t},\frac{\partial\hat{f}}{\partial t}\right) & =1,\label{eq:new_metric_unit_length}\\
\hat{g}\left(\frac{\partial\hat{f}}{\partial t},\frac{\partial\hat{f}}{\partial y_{k}}\right) & =0,\text{ and}\label{eq:new_metric_perpendicular}\\
\hat{g}\left(\frac{\partial\hat{f}}{\partial p_{k}},\frac{\partial\hat{f}}{\partial p_{l}}\right) & =[1-h(t)]g\left(\frac{\partial f_{1}}{\partial p_{k}},\frac{\partial f_{1}}{\partial p_{l}}\right)+h(t)g\left(\frac{\partial f_{2}}{\partial p_{k}},\frac{\partial f_{2}}{\partial p_{l}}\right),\nonumber \\
 & {\rm \ for\ }1\le k,l\le n-1.\nonumber \end{align}

We know that, for $i=0,1,2$, the original metric $g$ satisfies

\begin{align*}
g\left(\frac{\partial f_{i}}{\partial t},\frac{\partial f_{i}}{\partial t}\right) & =1,{\rm \ and}\\
g\left(\frac{\partial f_{i}}{\partial t},\frac{\partial f_{i}}{\partial p_{k}}\right) & =0,{\rm \ for\ }k=1,\dots,n-1,\end{align*}
 in the region $f_{i}(R_{0,0})$.

We define the required Riemannian metric as \[
\tilde{g}=H\hat{g}+(1-H)g,\]
 where we interpret $H\hat{g}$ to be $0$ when $H=0$.

If $(t,\mathbf{p})\in[-a_{1},-a_{5}]\times B_{b_{1}}$, then $h(t)=0$
and $\phi(t,\mathbf{p})=(t,\mathbf{p})$; if $(t,\mathbf{p})\in[a_{5},a_{1}]\times B_{b_{1}}$,
then $h(t)=1$ and $\phi(t,\mathbf{p})=f_{1}^{-1}\circ f_{2}(t,\mathbf{p})$.
Thus

\begin{equation}
\hat{f}(t,\mathbf{p})=\begin{cases}
f_{1}(t,\mathbf{p}), & \text{if $(t,{\bf \mathbf{p}})\in[-a_{1},-a_{5}]\times B_{b_{1}}$};\\
f_{2}(t,\mathbf{p}), & \text{if $(t,\mathbf{p})\in[a_{5},a_{1}]\times B_{b_{1}}$}.\end{cases}\label{eq:blending}\end{equation}
Therefore $\hat{g}$ agrees with $g$ on $\hat{f}(R_{2,2}\setminus R_{5,2})$.
If $f_{1}$ and $f_{2}$ are close to $f_{0}$ in $C^{\infty}(R_{0,0},M)$
, then $\hat{g}$ is $C^{\infty}$-close to $g$ on $\hat{f}(R_{2,2})\supseteq f_{0}(\overline{R_{3,3}})$.
Since $\tilde{g}=g$ on $M\setminus f_{0}(\overline{R_{3,3}})$, we
may choose $\mathcal{F}_{0}$ sufficiently small so that $\tilde{g}\in\mathcal{N}$
for $f_{1},f_{2}\in\mathcal{F}_{0}$.

To summarize, we have chosen $\mathcal{F}_{0}$ sufficiently small
so that if $f_{1},f_{2}\in\mathcal{F}_{0}$, then (\ref{eq:rectangle-containment}),(\ref{eq:f-hat-is-diffeo}),(\ref{eq:f-hat-rectangle}),(\ref{eq:f1_fhat}),
and (\ref{eq:f1_f2_fhat}) hold, and $\tilde{g}\in\mathcal{N}$.

Now we verify that (1), (2), and (3) hold.

The region where $\hat{g}$ is defined and not equal to $g$ is contained
in $\hat{f}(R_{5,2})$, which is a subset of $f_{1}(R_{3,1})\cap f_{2}(R_{3,1})$,
by (\ref{eq:f-hat-rectangle}). Therefore $\tilde{g}=g$ on the complement
of $f_{1}(R_{3,1})\cap f_{2}(R_{3,1})$, which is conclusion $(1)$.

Since $H=1$ on $f_{0}(\overline{R_{4,4}})\supseteq\hat{f}(R_{5,5})$,
we have $\tilde{g}=\hat{g}$ on $\hat{f}(R_{5,5})$. For each $\mathbf{p}\in B_{b_{5}}$,
we define a curve $\gamma_{\mathbf{p}}:I_{a}\rightarrow M$ as \[
\gamma_{\mathbf{p}}(t)=\begin{cases}
f_{1}(t,\mathbf{p}), & \text{if $t\in(-a,-a_{2}]$};\\
\hat{f}(t,\mathbf{p}), & \text{if $t\in(-a_{2},a_{2})$};\\
f_{2}(t,\mathbf{p}), & \text{if $t\in[a_{2},a)$}.\end{cases}\]
 It follows from (\ref{eq:blending}) that these curves are smooth.
Moreover, these curves are $\tilde{g}$-geodesics, because $\tilde{g}=g$
on $f_{1}((-a,-a_{2}]\times B_{b_{5}})\cup f_{2}([a_{2},a)\times B_{b_{5}})$
(by (\ref{eq:f1_f2_fhat})), $\hat{g}=g=\tilde{g}$ on $\hat{f}((\overline{I_{a_{2}}}\setminus I_{a_{5}})\times B_{b_{5}})=f_{1}([-a_{2},-a_{5}]\times B_{b_{5}})\cup f_{2}([a_{5},a_{2}]\times B_{b_{5}})$,
$\tilde{g}=\hat{g}$ on $\hat{f}(R_{5,5})$, and the curves $t\mapsto\hat{f}(t,\mathbf{p})$
are $\hat{g}$-geodesics for all $\mathbf{p}\in B_{b_{2}}$ (by (\ref{eq:new_metric_unit_length})
and (\ref{eq:new_metric_perpendicular})). This proves conclusion
(2). If $f_{1}(t,\mathbf{0})=f_{2}(t,\mathbf{0})$ for $t\in I_{a}$,
then $\phi(t,\mathbf{0})=(t,\mathbf{0})$ and $\hat{f}(t,\mathbf{0})=f_{1}(t,\mathbf{0})$
for $t\in\overline{I_{a_{1}}}$. Therefore the $\tilde{g}$-geodesic
$\gamma_{\mathbf{0}}$ is the same as $t\mapsto f_{1}(t,\mathbf{0})$,
which establishes (3).

\end{proof}
We now define a notion of ${\it merging}$ for two geodesics. This
will be used in Lemma \ref{lem:geodesic_merging} below.
\begin{defn}
\label{definition} Let $M$ be a $C^{\infty}$-manifold, and let
$g,\tilde{g}$ be Riemannian metrics on $M$. Suppose $U$ is an open
set in $M$, $t_{0}\in\mathbb{R}$, and $\gamma_{i}:[\hat{r}_{i},\hat{s}_{i}]\rightarrow M$,
$i=1,2$, are $g$-geodesics such that \begin{equation}
\{t\in[\hat{r}_{i},\hat{s}_{i}]:\gamma_{i}(t)\in U\}=(r_{i},s_{i}),{\rm {\rm {\rm \ where\ }}\hat{r}_{i}<r_{i}<t_{0}<s_{i}<\hat{s}_{i}.}\label{eq:merging_interval}\end{equation}
We say that a $\tilde{g}$-geodesic $\gamma:[\hat{r}_{1},\hat{s}_{2}]\rightarrow M$,
\emph{merges} $\gamma_{1}$ \emph{and} $\gamma_{2}$ \emph{within}
$U$ if there exist $\tilde{r},\tilde{s}$ such that $r_{1}<\tilde{r}<t_{0}<\tilde{s}<s_{2}$,
$\gamma(\tilde{r},\tilde{s})\subseteq U$, $\gamma(t)=\gamma_{1}(t)$
for $\hat{r}_{1}\leq t\leq\tilde{r}$, and $\gamma(t)=\gamma_{2}(t)$
for $\tilde{s}\leq t\leq\hat{s}_{2}$.
\end{defn}
The following lemma allows us to merge two geodesics according to
Definition \ref{definition}. K. Burns and G. Paternain have a similar
result in the 2-dimensional case \cite{bib2}.
\begin{lem}
\label{lem:geodesic_merging}Let $(M,g)$ be a closed $C^{\infty}$
Riemannian manifold of dimension $n\ge2$, and let $\mathcal{N}$
be an open neighborhood of $g$ in the $C^{\infty}$ topology. Suppose
$U$ is a convex (with respect to $g)$ open set in $M$ and $(x_{0},v_{0})\in T^{1}U.$
Then there exists an open neighborhood $\mathcal{V}$ of $(x_{0},v_{0})$
in $T^{1}U$ such that for any $(x_{i},v_{i})\in\mathcal{V}$, $i=1,2$,
if $\gamma_{i}:[\hat{r}_{i},\hat{s}_{i}]\to M$ are $g$-geodesics
that satisfy (\ref{eq:merging_interval}) and $(\gamma_{i}(t_{0}),\gamma_{i}'(t_{0}))=(x_{i},v_{i})$,
for $i=1,2$, then there exists $\tilde{g}\in\mathcal{N}$ which agrees
with $g$ on $M\setminus U$, and a $\tilde{g}$-geodesic $\gamma$
that merges $\gamma_{1}$ and $\gamma_{2}$ within $U$. \end{lem}
\begin{proof}
Let $\gamma_{0}:[\hat{r}_{0},\hat{s}_{0}]\to M$ be a $g$-geodesic
such that $(\gamma_{0}(t_{0}),\gamma_{0}'(t_{0}))=(x_{0},v_{0})$
and (\ref{eq:merging_interval}) is satisfied for $i=0$ and some
choice of $r_{0},s_{0}$. By replacing $U$ by a smaller convex open
neighborhood of $x_{0}$, if necessary, we may assume there exist
$C^{\infty}$ orthonormal vector fields $E_{1},\dots,E_{n}$ on $U$
such that $E_{n}(\gamma_{0}(t))=\gamma_{0}'(t)$ for all $t\in(r_{0},s_{0})$.
We may assume that $t_{0}=0.$ Choose $T$ such that $0<T<|r_{0}|$
and $\tilde{x}_{0}:=\gamma_{0}(-T)$ is not conjugate to $x_{0}$
along $\gamma_{0}|[-T,0].$ For $u\in U$ and $\mathbf{z}=(z_{1},\dots,z_{n})\in\mathbb{R}^{n},$
let \begin{equation}
\Phi(u,{\bf \mathbf{z})}=z_{1}E_{1}(u)+\cdots+z_{n}E_{n}(u)\in T_{u}U.\label{eq:Phi_def}\end{equation}
 Define $\varphi:\{\mathbf{p}=(p_{1},\dots p_{n-1})\in\mathbb{R}^{n-1}:|\mathbf{p}|<1\}\to\{\mathbf{w}\in\mathbb{R}^{n}:|\mathbf{w}|=1\}$
by

\begin{equation}
\varphi(\mathbf{p})=(p_{1},\dots p_{n-1},1-(p_{1}^{2}+\cdots+p_{n-1}^{2})^{1/2}).\label{eq:varphi_def}\end{equation}
 Since $\tilde{x}_{0}$ and $x_{0}$ are not conjugate along $\gamma_{0}|[-T,0]$
, there exist $\tilde{a},\tilde{b}>0$ such that the map\[
f_{0}(t,\mathbf{p}):={\rm exp}_{\tilde{x}_{0}}(\Phi(\tilde{x}_{0},(t+T)\varphi(\mathbf{p}))),\]
defined for $(t,\mathbf{p})\in I_{\tilde{a}}\times B_{\tilde{b}}$,
is a $C^{\infty}$ diffeomorphism onto its image, and its image is
contained in $U.$ Note that $f_{0}(0,\mathbf{0})=x_{0}.$ Moreover,
there exist $a,b$ with $0<a<\tilde{a}$, $0<b<\tilde{b}$, an open
neighborhood $\mathcal{A}$ of $Id$ in $SO(n),$ and an open neighborhood
$\tilde{U}$ of $\tilde{x}_{0}$ in $U$ such that for $\tilde{x}\in\tilde{U}$
and $A\in\mathcal{A}$, the map\[
f(t,{\bf p}):={\rm exp}_{\tilde{x}}(\Phi(\tilde{x},(t+T)A(\varphi({\bf p})))),\]
defined for $(t,{\bf p})\in I_{a}\times B_{b}$ is a $C^{\infty}$
diffeomorphism onto its image, and its image is in $U$. Now choose
$\mathcal{V}$ to be an open neighborhood of $(x_{0},v_{0})$ in $T^{1}U$
such that for each $(x,v)\in\mathcal{V}$, the geodesic $\tilde{\gamma}$
with $(\tilde{\gamma}(0),\tilde{\gamma}'(0))=(x,v)$ satisfies $\tilde{x}:=\tilde{\gamma}(-T)\in\tilde{U}$
and there exists $A\in\mathcal{A}$ with $\Phi(\tilde{x},A(\varphi(\mathbf{0})))=\tilde{\gamma}'(-T).$
We also require $\mathcal{V}$ to be small enough so that $\tilde{x}$
is sufficiently close to $\tilde{x}_{0}$ and $A$ can be chosen sufficiently
close to $Id$, so that $f$ is in the neighborhood $\mathcal{F}_{0}$
of $f_{0}$ given in Lemma \ref{foliation merging}. (The hypothesis
(iii) in Lemma \ref{foliation merging} for $f_{0}$, as well as $f_{1},f_{2}$
defined below, follows from the Gauss Lemma.)

Let $(x_{i},v_{i})\in\mathcal{V},$ $i=1,2,$ and suppose $\gamma_{i}:[\hat{r}_{i},\hat{s}_{i}]\rightarrow M$,
$i=1,2$, are $g$-geodesics such that (\ref{eq:merging_interval})
is satisfied and $(\gamma_{i}(0),\gamma_{i}'(0))=(x_{i},v_{i}).$
Let $r_{i},s_{i},$ $i=1,2,$ be as in (\ref{eq:merging_interval}).
For $i=1,2,$ define \[
f_{i}(t,{\bf \mathbf{p}}):={\rm exp}_{\tilde{x}_{i}}(\Phi(\tilde{x}_{i},(t+T)A_{i}(\varphi(\mathbf{p})))),\]
for $(t,\mathbf{p})\in I_{a}\times B_{b},$ where $\tilde{x}_{i}:=\gamma_{i}(-T),$
and $A_{i}\in\mathcal{A}$ is such that $\Phi(\tilde{x}_{i},A_{i}(\varphi(\mathbf{0})))=\gamma_{i}'(-T).$
Then $f_{i}(t,\mathbf{0})=\gamma_{i}(t)$ for $t\in I_{a}.$ From
Lemma \ref{foliation merging}, we obtain $\tilde{g}\in\mathcal{N}$
which agrees with $g$ on $M\setminus U$ so that conclusion (2) of
Lemma \ref{foliation merging} holds. Finally, we define the required
$\tilde{g}$-geodesic $\gamma:[\hat{r}_{1},\hat{s}_{2}]\rightarrow M$
as \[
\gamma(t)=\begin{cases}
\gamma_{1}(t), & \text{if $t\in[\hat{r}_{1},-a]$};\\
\gamma_{\mathbf{0}}(t), & \text{if $t\in(-a,a)$};\\
\gamma_{2}(t), & \text{if $t\in[a,\hat{s}_{2}]$},\end{cases}\]
 where $\gamma_{\mathbf{0}}$ is as in Lemma \ref{foliation merging}(2).
\end{proof}
Lemma \ref{lem:nonconjugacy} below allows us to destroy conjugate
points along a geodesic by making a small perturbation of the metric.
A two-dimensional version of this lemma is contained in \cite{bib2}.
\begin{lem}
\label{lem:nonconjugacy} Let $(M,g)$ be a closed $C^{\infty}$ Riemannian
manifold of dimension $n\ge2$, and let $\mathcal{N}$ be an open
neighborhood of $g$ in the $C^{\infty}$ topology. Let $x,y\in M$
and suppose $\gamma:[0,L]\to M$ is a $g$-geodesic from $x$ to $y$.
Let $0=t_{0}<t_{1}<\cdots<t_{\ell}=L$, where $\ell\ge1$, and define
$z_{k}:=\gamma(t_{k})$ for $k=0,\dots,\ell.$ Suppose $s_{0}\in(t_{j},t_{j+1})$
for some $j\in\{0,\dots,\ell-1\}$ and $u_{0}:=\gamma(s_{0})$ is
not a self-intersection point of $\gamma$ (i.e., $u_{0}\notin\gamma([0,T]\setminus\{s_{0}\})$).
Let $U_{0}$ be an open neighborhood of $u_{0}.$ Then there exists
$\hat{g}\in\mathcal{N}$ that agrees with $g$ on $M\setminus U_{0}$
such that the following conditions hold:
\begin{enumerate}
\item $\gamma$ is also a unit speed geodesic for $\hat{g}$.
\item If $k_{1}$ and $k_{2}$ are integers such that $0\le k_{1}\le j$
and $j+1\le k_{2}\le\ell$, then $z_{k_{1}}$ is not conjugate to
$z_{k_{2}}$ along $\gamma|[t_{k_{1}},t_{k_{2}}]$ in the $\hat{g}$
metric.
\end{enumerate}
\end{lem}
\begin{proof}
It suffices to prove the lemma for the case $\ell=1$ and $0=t_{0}<s_{0}<t_{1}=L$,
because we can then obtain (2) in the general case through a finite
sequence of perturbations of the metric (within $\mathcal{N}$) corresponding
to each possible pair $(k_{1},k_{2})$ with $0\le k_{1}\le j$ and
$j+1\le k_{2}\le\ell$. Each successive perturbation adds one more
pair $(k_{1},k_{2})$ such that $z_{k_{1}}$ is not conjugate to $z_{k_{2}}$
along $\gamma|[t_{k_{1}},t_{k_{2}}]$, and the perturbations can be
taken small so that no new conjugacies are introduced between such
pairs of points.

We now assume $\ell=1$ and $0=t_{0}<s_{0}<t_{1}=L$ . By perturbing
$s_{0}$ slightly, if necessary, we may assume that $x$ is not conjugate
to $u_{0}$ along $\gamma|[0,s_{0}]$. We may also assume that the
open neighborhood $U_{0}$ of $u_{0}$ is chosen so that $\{t\in[0,L]:\gamma(t)\in U_{0}\}=(s_{0}-\eta,s_{0}+\eta)$
for some $\eta$ with $0<\eta<\min(s_{0},t_{1}-s_{0})$. Let $U$
be an open neighborhood of $x$ disjoint from $U_{0}$. Suppose $\tau\in(0,s_{0}-\eta)$
is such that $\gamma|[0,\tau]$ is one-to-one, and whenever $0<t\le\tau,$
$x$ is not conjugate to $\gamma(t)$ along $\gamma|[0,t]$, and $\gamma(t)$
is not conjugate to $y$ along $\gamma|[t,L].$ Let $E_{1},\dots,E_{n}$
be $C^{\infty}$ vector fields along $\gamma|[0,\tau]$ with $\gamma'(t)=E_{n}(\gamma(t))$
for $t\in[0,\tau].$ Let $\Phi$ and $\varphi$ be as in (\ref{eq:Phi_def})
and (\ref{eq:varphi_def}) for $u\in\gamma([0,\tau]).$ Since $x$
is not conjugate to $u_{0}$ along $\gamma|[0,s_{0}],$ there exist
$\tilde{a},\tilde{b}>0$ such that the map \[
f_{1}(t,\mathbf{p}):=\exp_{x,g}(\Phi(x,(t+s_{0})\varphi(\mathbf{p})),\]
defined for $(t,\mathbf{p})\in I_{\tilde{a}}\times B_{\tilde{b}}$,
is a $C^{\infty}$ diffeomorphism onto its image, and its image is
in $U_{0}.$ (The `$g$' in the subscript indicates we are referring
to the exponential map for the metric $g.$ ) There exist $a,b,\tilde{\delta}$
with $0<a<\tilde{a,}$ $0<b<\tilde{b,}$ $0<\tilde{\delta}<\tau$,
such that the map \[
f_{2}(t,\mathbf{p}):=\exp_{\tilde{x},g}(\Phi(\tilde{x},(t+s_{0}-\delta)\varphi(\mathbf{p})),\]
defined for $(t,\mathbf{p})\in I_{a}\times B_{b}$ is a $C^{\infty}$
diffeomorphism onto its image, and its image is in $U_{0}$ for any
$\tilde{x}:=\gamma(\delta)$ with $0<\delta<\tilde{\delta.}$ Let
$f_{0}$ be the restriction of $f_{1}$ to $I_{a}\times I_{b}$, and
let $\mathcal{F}_{0}$ be as in Lemma \ref{foliation merging}. We
choose $\delta$ sufficiently small so that $f_{2}\in\mathcal{F}_{0}.$
Since $f_{1}(I_{a/2}\times B_{b/2})\cap f_{2}(I_{a/2}\times B_{b/2})$
is a subset of $U_{0}$, Lemma \ref{foliation merging} implies that
there is a $\hat{g}\in\mathcal{N}$ which agrees with $g$ on $M\setminus U_{0}$
and Lemma \ref{foliation merging}(2) holds with $\tilde{g}$ replaced
by $\hat{g}.$ We also obtain Lemma \ref{foliation merging}(3) with
$\tilde{g}$ replaced by $\hat{g}$, because $f_{1}(t,\mathbf{0})=f_{2}(t,\mathbf{0})$
for $t\in I_{a}.$ Therefore $\gamma$ is also a geodesic for $\hat{g}$.
For $\mathbf{p}\in B_{b/4}$, let $\gamma_{\mathbf{p}}$ be as in
Lemma \ref{foliation merging}(2) and define $\sigma_{\mathbf{p}}:[0,L]\to M$
by \begin{equation}
\sigma_{\mathbf{p}}(t)=\begin{cases}
\exp_{x,g}(\Phi(x,t\varphi(\mathbf{p})), & \text{if $t\in[0,s_{0}-a]$};\\
\gamma_{\mathbf{p}}(t-s_{0}), & \text{if $t\in(s_{0}-a,s_{0}+a)$};\\
\exp_{\tilde{x},g}(\Phi(\tilde{x},(t-\delta)\varphi(\mathbf{p})), & \text{if $t\in[s_{0}+a,L]$}.\end{cases}\label{eq:sigma_definition}\end{equation}
 Then $\sigma_{\mathbf{p}}$ is a $\hat{g}$-geodesic that merges,
within $U_{0}$, a $g$-geodesic originating at $x$ with initial
velocity $\Phi(x,\varphi(\mathbf{p}))$ and a $g$-geodesic that is
at $\tilde{x}$ with velocity $\Phi(\tilde{x},\varphi(\mathbf{p}))$
at time $\delta.$ Thus, for $\mathbf{p}\in B_{b/4},$ \begin{equation}
\exp_{x,\hat{g}}(\Phi(x,t\varphi(\mathbf{p})))=\exp_{\tilde{x},g}(\Phi(x,(t-\delta)\varphi(\mathbf{p}))\label{eq:exp_translation}\end{equation}
for $s_{0}+a\le t\le L.$ Since $\tilde{x}$ is not conjugate to $y$
along $\gamma|[\delta,L]$ in the metric $g$, $\exp_{\tilde{x},g}$
is locally a diffeomorphism near $(L-\delta)\gamma'(\delta).$ By
(\ref{eq:exp_translation}), this implies that $\exp_{x,\hat{g}}$
is locally a diffeomorphism near $L\gamma'(0)$. Therefore $x$ is
not conjugate to $y$ along $\gamma$ in the $\hat{g}$ metric.
\end{proof}
A ${\it geodesic\ lasso}$ is defined to be a closed curve which is
a geodesic except at one point, where it fails to be regular. The
following Lemma \ref{lem:avoiding_Z} allows us to perturb a geodesic
so that it avoids a finite set of points on $M$, and it also allows
us to change a closed geodesic to a geodesic lasso.
\begin{lem}
\label{lem:avoiding_Z} Let $(M,g)$ be a closed $C^{\infty}$ Riemannian
manifold of dimension $n\ge2$, and let $\mathcal{N}$ be an open
neighborhood of $g$ in the $C^{\infty}$ topology. Let $x,y\in M$
and suppose $\gamma:[0,L]\to M$ is a $g$-geodesic from $x$ to $y$.
Let $Z$ be a finite set of points in $M$ such that $x,y\in Z$.
Let $\{t\in[0,L]:\gamma(t)\in Z\}=\{t_{k}:k=0,\dots,\ell\},$ where
$0=t_{0}<\cdots<t_{\ell}=L$, $\ell\ge1$, and define $z_{k}:=\gamma(t_{k})$,
for $k=0,\dots\ell.$ Assume that
\begin{enumerate}
\item [(i)]$x$ is not conjugate to $z_{k}$ along $\gamma|[0,t_{k}],$
for $k=1,\dots,\ell.$
\item [(ii)] $z_{k}$ is not conjugate to $y$ along $\gamma|[t_{k},L]$,
for $k=0,\dots,\ell-1.$
\end{enumerate}
Suppose $s_{0}\in(0,L)$, $u_{0}:=\gamma(s_{0})$ is not a self-intersection
point of $\gamma$, and $u_{0}\notin Z$. Let $U_{0}$ be an open
neighborhood of $u_{0}.$ Then there exist open neighborhoods $W_{1}$
and $W_{2}$ of $\gamma'(0)$ and $\gamma'(L)$ in $T_{x}^{1}M$ and
$T_{y}^{1}M$ , respectively, such that for any $w_{1}\in W_{1}\setminus\{\gamma'(0)\}$
and any $w_{2}\in W_{2}\setminus\{\gamma'(L)\}$, there exists $\tilde{g}\in\mathcal{N}$
that agrees with $g$ on $M\setminus U_{0}$ and a $\tilde{g}$-geodesic
$\tilde{\gamma}:[0,L]\to M$ from $x$ to $y$ such that $\tilde{\gamma}'(0)=w_{1},$
$\tilde{\gamma}'(L)=w_{2}$, $\tilde{\gamma}((0,L))\cap Z=\emptyset$,
and $x$ is not conjugate to $y$ along $\tilde{\gamma}$ for $\tilde{g}$.

\end{lem}
\begin{proof}
We may assume that $Z\subset\gamma([0,L])$. By replacing $U_{0}$
by a smaller open neighborhood of $u_{0}$ if necessary, we may assume
that $U_{0}$ is convex for $g$, $U_{0}\cap Z=\emptyset,$ and $\{t\in[0,L]:\gamma(t)\in U_{0}\}=(s_{0}-\eta,s_{0}+\eta),$
for some $\eta>0.$

Since $x$ is not conjugate to $z_{k}$ along $\gamma|[0,t_{k}]$
for $k=1,\dots,\ell,$ and $\exp_{x,g}$ is locally a diffeomorphism
near $0\in T_{x}M$, there exist neighborhoods $V_{k}$ of $t_{k}\gamma'(0)$
in $T_{x}M$, for $k=0,\dots,\ell$, such that the maps $\exp_{x,g}:V_{k}\to M$
are diffeomorphisms onto their images. Also,\begin{equation}
Z\cap\exp_{x,g}(\{t\gamma'(0):t\in[0,L]\}\setminus(V_{0}\cup\cdots V_{\ell}))=\emptyset,\label{eq:Z_intersect_exp}\end{equation}
because $(\exp_{x,g}^{-1}Z)\cap\{t\gamma'(0):t\in[0,L]\}=\{t_{0}\gamma'(0),\dots,t_{\ell}\gamma'(0)\}.$
By the continuity of $\exp_{x,g}$, we can choose $W_{1}$ sufficiently
small so that (\ref{eq:Z_intersect_exp}) still holds for $\gamma$
replaced by any $g$-geodesic $\gamma_{1}:[0,L]\to M$ with $\gamma_{1}(0)=x$
and $\gamma_{1}'(0)\in W_{1}.$ If $\gamma_{1}'(0)\in W_{1}\setminus\{\gamma'(0)\},$
then $\{t\gamma_{1}'(0):t\in(0,L]\}\cap(V_{0}\cup\cdots\cup V_{\ell})$
does not contain any of $t_{k}\gamma'(0),$ $k=0,\dots,\ell.$ Thus,
(\ref{eq:Z_intersect_exp}) for $\gamma_{1}$ implies that $\gamma_{1}((0,L])\cap Z=\emptyset.$
Similarly, if $W_{2}$ is sufficiently small, then for any $g$-geodesic
$\gamma_{2}:[0,L]\to M$ with $\gamma_{2}(L)=y$ and $\gamma_{2}'(L)\in W_{2}\setminus{\gamma'(L)}$
, we have $\gamma_{2}([0,L))\cap Z=\emptyset.$

Let $v_{0}=\gamma'(s_{0})$ and let $\mathcal{V}$ be an open neighborhood
of $(u_{0},v_{0})$ in $T^{1}U_{0}$ satisfying the conclusion of
Lemma \ref{lem:geodesic_merging} (with $U$ replaced by $U_{0}$
and $x_{0}$ replaced by $u_{0}$). In addition to the requirements
of the preceding paragraph, we require $W_{1}$ and $W_{2}$ to be
sufficiently small so that if $\gamma_{i}:[0,L]\to M$, $i=1,2$,
are such that $\gamma_{1}(0)=x$, $\gamma_{1}'(0)\in W_{1}$, $\gamma_{2}(L)=y$,
and $\gamma_{2}'(L)\in W_{2},$ then there exist $r_{i},s_{i}$ with
$0<r_{i}<s_{0}<s_{i}<L$, such that $\{t\in[0,L]:\gamma_{i}(t)\in U_{0}\}=(r_{i},s_{i})$
and $(\gamma_{i}(s_{0}),\gamma_{i}'(s_{0}))\in\mathcal{V}$.

Suppose $w_{1}\in W_{1}\setminus\{\gamma'(0)\}$ and $w_{2}\in W_{2}\setminus\{\gamma'(L)\}$,
and let $\gamma_{i}:[0,L]\to M$, $i=1,2$, be $g$-geodesics such
that $\gamma_{1}(0)=x$, $\gamma_{1}'(0)=w_{1}$, $\gamma_{2}(L)=y$,
and $\gamma_{2}'(L)=w_{2}$. By Lemma \ref{lem:geodesic_merging},
there exists a metric $\tilde{g}\in\mathcal{N}$ that agrees with
$g$ on $M\setminus U_{0}$ and a $\tilde{g}$-geodesic $\tilde{\gamma}:[0,L]\to M$
that merges $\gamma_{1}$ and $\gamma_{2}$ within $U_{0}.$ Since
$U_{0}\cap Z=\emptyset$ and $\gamma_{1}((0,L))\cap Z=\emptyset=\gamma_{2}((0,L))\cap Z$,
we have $\tilde{\gamma}((0,L))\cap Z=\emptyset$. By Lemma \ref{lem:nonconjugacy}
we can make a small additional perturbation of the metric $\tilde{g}$
within $U_{0}$, if necessary, to arrange for $x$ and $y$ to not
be conjugate along $\tilde{\gamma}.$
\end{proof}

\section{Proof of Theorem \ref{thm:main}}

We now use the results of Section 2 to prove Theorem \ref{thm:main}.
The notation tr$(\gamma)$ will mean the trace of a curve $\gamma:I\to M$,
i.e., tr$(\gamma)=\{\gamma(t):t\in I\}.$
\begin{proof}
Let $(x,y,g)\in M\times M\times\mathbb{G}$, and let $n\in\mathbb{N}$.
We consider the statement $S(x,y,n,g):$ there exist $g$-geodesics
$\gamma_{i}:[0,L_{i}]\rightarrow M$ from $x$ to $y$, $i=1,\dots,n$,
which satisfy the following four properties:
\begin{enumerate}
\item [(i)]If $x\neq y$, then the set of tangent vectors \[
\{\gamma_{1}'(0),\gamma_{2}'(0),\dots,\gamma_{n}'(0)\}\]
 at $x$ are pairwise linearly independent, and the set of tangent
vectors \[
\{\gamma_{1}'(L_{1}),\gamma_{2}'(L_{2}),\dots,\gamma_{n}'(L_{n})\}\]
 at $y$ are pairwise linearly independent. If $x=y$, then the set
of tangent vectors \[
\{\gamma_{1}'(0),\gamma_{1}'(L_{1}),\gamma_{2}'(0),\gamma_{2}'(L_{2}),\dots,\gamma_{n}'(0),\gamma_{n}'(L_{n})\}\]
 are pairwise linearly independent. Thus we cannot join $\gamma_{i}$
to $\gamma_{j}$ smoothly at $x$ or at $y$, for any $i,j\in\{1,\dots,n\}$.
\item [(ii)]For each $i=1,\dots,n$, we have $\gamma_{i}((0,L_{i}))\cap\{x,y\}=\emptyset$.
That is, $\gamma_{i}$ meets $x$ and $y$ only at its endpoints.
\item [(iii)]Any three of $\gamma_{1},\dots,\gamma_{n}$ are concurrent
only at $x$ and at $y$.
\item [(iv)]The point $x$ is not conjugate to $y$ in the metric $g$
along $\gamma_{i}|[0,L_{i}]$, for $i=1,\dots,n$.
\end{enumerate}
We define $\mathcal{H}_{n}(x,y):=\{g\in\mathbb{G}:S(x,y,n,g){\rm \ is\ satisfied}\}$.
We make the following claim:
\begin{claim}
\label{claim1} $(a)$ $\mathcal{H}_{n}(x,y)$ is $C^{\infty}$-dense
in $\mathbb{G}$ and $(b)$ there is a $C^{1}$-open neighborhood
$\mathcal{G}_{n}(x,y)$ of $\mathcal{H}_{n}(x,y)$ in $\mathbb{G}$
such that parts (i), (ii), and (iii) of $S(x,y,n,g)$ are satisfied
for all $g\in\mathcal{G}_{n}(x,y)$.
\end{claim}
Claim \ref{claim1} implies that the set $\bigcap\mathcal{G}_{n}(x,y)$
is a dense $G_{\delta}$ set for $\mathbb{G}$ with the $C^{k}$ topology,
for $k=1,2,\dots,\infty.$ Suppose $P\subseteq M\setminus\{x,y\}$
is a set with $m$ points, and $g\in\bigcap\mathcal{G}_{n}(x,y)$.
Since $g\in\mathcal{G}_{2m+1}$, we can find $2m+1$ $g$-geodesics
that satisfy (iii). If $P$ were a blocking set for $(x,y)$, then
by the pigeonhole principle, at least three of these geodesics would
pass through the same point in $P$, which leads to contradiction.
Hence there is no finite blocking set for $(x,y)$, and Theorem \ref{thm:main}(1)
follows from Claim \ref{claim1}.

Similarly, if we define $\widetilde{\mathcal{H}}_{n}:=\{(x,y,g)\in M\times M\times\mathbb{G}:S(x,y,n,g){\rm \ is\ satisfied}\}$
and $\widehat{\mathcal{H}}_{n}:=\{(x,g)\in M\times\mathbb{G}:S(x,x,n,g)\}{\rm \ is\ satisfied}\}$,
and we prove the following claims, then Theorem \ref{thm:main}(2),(3)
will follow by considering $\bigcap\widetilde{\mathcal{G}}_{n}$ and
$\bigcap\widehat{\mathcal{G}}_{n}$, respectively.
\begin{claim}
\label{claim2} $(a)$ $\widetilde{\mathcal{H}}_{n}$ is $C^{\infty}$-dense
in $M\times M\times\mathbb{G}$ and $(b)$ there is a $C^{1}$-open
neighborhood $\widetilde{\mathcal{G}}_{n}$ of $\widetilde{\mathcal{H}}_{n}$
in $M\times M\times\mathbb{G}$ such that (i), (ii), and (iii) of
$S(x,y,n,g)$ are satisfied for all $(x,y,g)\in\widetilde{\mathcal{H}}_{n}$.
\begin{claim}
\label{claim3} $(a)$ $\widehat{\mathcal{H}}_{n}$ is $C^{\infty}$-dense
in $M\times\mathbb{G}$ and $(b)$ there is a $C^{1}$-open neighborhood
$\widehat{\mathcal{G}}_{n}$ of $\widehat{\mathcal{H}}_{n}$ in $M\times\mathbb{G}$
such that (i), (ii), and (iii) of $S(x,x,n,g)$ are satisfied for
all $(x,g)\in\widehat{\mathcal{H}}_{n}$.
\end{claim}
\end{claim}
We now prove Claim \ref{claim1}$(a)$ by mathematical induction.
For $n=1$, let $\mathcal{N}$ be any non-empty $C^{\infty}$-open
set in $\mathbb{G}$, and let $g\in\mathcal{N}$. Let $\gamma:[0,L]\to M$
be a $g$-geodesic from $x$ to $y$. By restricting the domain of
$\gamma$, if necessary, we may assume that $\gamma((0,L))\cap\{x,y\}=\emptyset.$
Then we let $\ell=1$ and $0=t_{0}<t_{1}=L$ in Lemma \ref{lem:nonconjugacy}.
By Lemma \ref{lem:nonconjugacy}, there exists $\hat{g}\in\mathcal{N}$
such that $\gamma$ is also a unit speed geodesic for $\hat{g}$ and
$x$ is not conjugate to $y$ along $\gamma.$ If $(x,\gamma'(0))\neq(y,\gamma'(L))$,
then we let $g_{1}=\hat{g}$, and $\gamma_{1}=\gamma$. If $(x,\gamma'(0))=(y,\gamma'(L))$,
that is, $\gamma$ is a closed geodesic, then we apply Lemma \ref{lem:avoiding_Z}
to obtain $g_{1}\in\mathcal{N}$ and a $g_{1}$-geodesic lasso $\gamma_{1}:[0,L]\to M$
with $\gamma_{1}(0)=\gamma_{1}(L)=x$ but $\gamma_{1}'(0)\neq\gamma_{1}'(L)$,
and $x\notin\gamma_{1}((0,L))$. Then (i) , (ii), and (iv) are satisfied,
and (iii) is vacuous. Since $\mathcal{N}$ is arbitrary, $\mathcal{H}_{1}(x,y)$
is $C^{\infty}$-dense.

Next we suppose $\mathcal{H}_{n-1}(x,y)$ is $C^{\infty}$-dense for
some $n\geq2$, and we will prove that $\mathcal{H}_{n}(x,y)$ is
$C^{\infty}$-dense. Let $\mathcal{N}$ be any non-empty $C^{\infty}$-open
set in $\mathbb{G}$. There exist $g_{n-1}\in\mathcal{H}_{n-1}(x,y)\cap\mathcal{N}$
and $g_{n-1}$-geodesics $\gamma_{i}:[0,L_{i}]\to M$ from $x$ to
$y$, $i=1,\dots,n-1$, so that properties (i) - (iv) are satisfied
with $n$ replaced by $n-1$. By Theorem \ref{Serre}, there exists
a $g_{n-1}$-geodesic $\gamma:[0,L]\to M$ from $x$ to $y$, distinct
from $\gamma_{1},\dots,\gamma_{n-1}$. If $x=y$, we also require
$\gamma$ to be distinct from $-\gamma_{1},\dots,-\gamma_{n-1},$
where $-\gamma_{i}$ is $\gamma_{i}$ traversed in the opposite direction.

By (i) and (ii), we have ${\rm tr}(\gamma)\nsubseteq{\rm tr}(\gamma_{1})\cup\cdots\cup{\rm tr}(\gamma_{n-1})$.
However, it may happen that ${\rm tr}(\gamma)$ contains one (or more)
of the sets ${\rm tr}(\gamma_{1}),\dots,{\rm tr}(\gamma_{n-1})$.
If $x=y$, then we can restrict the domain of $\gamma$, if necessary,
so that ${\rm tr}(\gamma)$ does not contain any of the sets ${\rm tr}(\gamma_{1}),\dots,{\rm tr}(\gamma_{n-1})$.
If $x\ne y$, then we can restrict the domain of $\gamma$, if necessary,
to obtain a $g_{n-1}$-geodesic from $x$ to $y$ such that one of
the following happens: (a) ${\rm tr}(\gamma)$ does not contain any
of the sets ${\rm tr}(\gamma_{1}),\dots,{\rm tr}(\gamma_{n-1})$;
(b) $\gamma$ consists of one of $\gamma_{1},\dots,\gamma_{n-1}$
preceded by a $g_{n-1}$-geodesic from $x$ to $x$; (c) $\gamma$
consists of one of $\gamma_{1},\dots,\gamma_{n-1}$ followed by a
$g_{n-1}$-geodesic from $y$ to $y$. If (a) holds, then we assume
that ${\rm tr}(\gamma)$ does not contain any of the sets ${\rm tr}(\gamma_{1}),\dots,{\rm tr}(\gamma_{n-1})$,
and the rest of this paragraph can be skipped. So assume that one
of cases (b) or (c) hold, and assume that the domain of $\gamma$
has been restricted so that cases (b) and (c) do not hold for any
further restriction to a proper closed subinterval of the domain.
Let $u_{0}\in{\rm tr}(\gamma)\setminus[{\rm tr}(\gamma_{1})\cup\cdots\cup{\rm tr}(\gamma_{n-1})]$
be such that $u_{0}$ is not a self-intersection point of $\gamma$,
and let $U_{0}$ be an open neighborhood of $u_{0}$ such that $U_{0}\cap[{\rm tr}(\gamma_{1})\cup\cdots\cup{\rm tr}(\gamma_{n-1})]=\emptyset$.
By Lemma \ref{lem:nonconjugacy}, we can make a perturbation of the
$g_{n-1}$ metric within $U_{0}$ such that $\gamma$ remains a geodesic,
the new metric is in $\mathcal{N}$, and neither of $x$ or $y$ is
conjugate to either of $x$ or $y$ along an arc of $\gamma.$ Then
Lemma \ref{lem:avoiding_Z} applies with $Z=\{x,y\}.$ Thus we may
again perturb the metric within $U_{0}$ to produce a new metric $\hat{g}\in\mathcal{N}$
and a $\hat{g}$-geodesics $\hat{\gamma}$ close to $\gamma$ and
different from $\gamma_{1},\dots,\gamma_{n-1}$, such that $\hat{\gamma}$
meets $x$ and $y$ only at its endpoints. In particular, ${\rm tr}(\gamma)$
does not contain any of the sets ${\rm tr}(\gamma_{1}),\dots,{\rm tr}(\gamma_{n-1})$.
Since $U_{0}\cap[{\rm tr}(\gamma_{1})\cup\cdots\cup{\rm tr}(\gamma_{n-1})]=\emptyset$,
$\gamma_{1},\dots,\gamma_{n-1}$ remain geodesics for $\hat{g}$.

From the preceding paragraph, we have a metric $\hat{g}\in\mathcal{N}$
and a $\hat{g}$-geodesic $\hat{\gamma}:[0,L]\to M$ from $x$ to
$y$ such that $\gamma_{1},\dots,\gamma_{n-1}$ are $\hat{g}$-geodesics
and ${\rm tr}(\hat{\gamma})$ does not contain any of the sets ${\rm tr}(\gamma_{1}),\dots,{\rm tr}(\gamma_{n-1})$.
Then ${\rm tr}(\hat{\gamma})\cap[{\rm tr}(\gamma_{1})\cup\cdots\cup{\rm tr}(\gamma_{n-1})]$
is a finite set. If $n=2$, let $Z=\{x,y\}$; if $n>2$, let $Z$
be the collection of all intersection points between the trace of
any two of $\gamma_{1},\dots,\gamma_{n-1}$. From (i) and (ii), we
know that $Z$ is a finite set. We also have $x,y\in Z$. We want
to perturb $\hat{\gamma}$ so that it does not meet $Z$ except at
its endpoints. Let $\hat{\gamma}^{-1}(Z)\cap[0,L]=\{t_{0},\dots,t_{l}\}$,
where $0=t_{0}<\dots<t_{\ell}=L$, and denote $z_{k}:=\hat{\gamma}(t_{k})$,
for $k=0,\dots,\ell.$ Let $s_{1}\in(t_{0},t_{1})$ , $s_{2}\in(t_{\ell-1},t_{\ell})$,
$s_{1}<s_{2}$, $u_{1}:=\hat{\gamma}(s_{1})$, $u_{2}:=\hat{\gamma}(s_{2})$
be such that $u_{1},u_{2}\notin{\rm tr}(\gamma_{1})\cup\cdots\cup{\rm tr}(\gamma_{n-1})$
and $u_{1},u_{2}$ are not self-intersection points of $\hat{\gamma}$.
We can apply Lemma \ref{lem:nonconjugacy} twice with $s_{0}=s_{i}$
and $U_{0}=U_{i}$ for $i=1,2$, where $(U_{1}\cup U_{2})\cap[{\rm tr}(\gamma_{1})\cup\cdots\cup{\rm tr}(\gamma_{n-1})]=\emptyset$.
Thus we obtain a metric $\bar{g}\in\mathcal{N}$ such that $\gamma_{1},\dots,\gamma_{n-1}$
are $\bar{g}$-geodesics, and conditions (i) and (ii) in Lemma \ref{lem:avoiding_Z}
hold for $g$ replaced by $\bar{g}$ and $\gamma$ replaced by $\hat{\gamma}$.
Hence, by Lemma \ref{lem:avoiding_Z}, there is a metric $\tilde{g}\in\mathcal{N}$
such that $\gamma_{1},\dots,\gamma_{n-1}$ are $\tilde{g}$-geodesics,
and there is a $\tilde{g}$-geodesic $\tilde{\gamma}$ from $x$ to
$y$ that is different from $\gamma_{1},\dots,\gamma_{n-1}$, and
does not meet any point of $Z$ except at its endpoints. Moreover,
by Lemma \ref{lem:avoiding_Z}, we may choose $\tilde{g}$ and $\tilde{\gamma}$
so that $x$ and $y$ are not conjugate along $\tilde{\gamma}$ in
the $\tilde{g}$-metric. All of the perturbations of the metric can
be done outside a neighborhood of ${\rm tr}(\gamma_{1})\cup\cdots\cup{\rm tr}(\gamma_{n-1})$.
We let $g_{n}=\tilde{g}.$ Then $\gamma_{1},\dots,\gamma_{n-1}$ are
$g_{n}$-geodesics, and (iv) remains true for $\gamma_{1},\dots,\gamma_{n-1}$
with the metric $g_{n}$. Thus properties (i)-(iv) hold for $\gamma_{1},\dots,\gamma_{n}$,
where $\gamma_{n}=\tilde{\gamma}$, and $g$ is replaced by $g_{n}$.
Since $\mathcal{N}$ is arbitrary, we conclude that $\mathcal{H}_{n}(x,y)$
is $C^{\infty}$-dense. This completes the proof of Claim \ref{claim1}(a).

Claim \ref{claim2}$(a)$ and Claim \ref{claim3}$(a)$ follow from
Claim \ref{claim1}(a), because $\tilde{\mathcal{H}}_{n}$ is $C^{\infty}$-
dense in each fiber $\{(x,y)\}\times\mathbb{G}$, and $\hat{\mathcal{H}}_{n}$
is $C^{\infty}$-dense in each fiber $\{x\}\times\mathbb{G}$.

Next we want to prove Claim \ref{claim1}$(b)$. Let $g\in\mathcal{H}_{n}(x,y)$,
and suppose $\gamma_{1},\dots,\gamma_{n}$ are $g$-geodesics that
satisfy properties (i)-(iv).

If we consider geodesics as curves in $T^{1}M$, then they are solutions
to a system of first order ordinary differential equations whose coefficients
depend only on the first derivatives of the metric. For the purpose
of defining $C^{1}$ distances between the given geodesics $\gamma_{i}$
and nearby curves $\tilde{\gamma}_{i},$ we extend the domain of $\gamma_{i}$
to $[0,L_{i}+1].$ The $C^{1}$ distance will be measured with respect
to the natural metric on $T^{1}M$ induced by $g$. For any $\epsilon>0$
there exists a $C^{1}$ neighborhood $\mathcal{N}_{1}$ of $g$ in
$\mathbb{G}$ and a $\delta=\delta(\epsilon)>0$ such that: if $\hat{g}\in\mathcal{N}_{1}$
and $\tilde{\gamma}_{i}$ is a $\hat{g}$-geodesic with $\tilde{\gamma_{i}}(0)=\gamma_{i}(0)$
and $|\tilde{\gamma_{i}}'(0)-\gamma_{i}'(0))|<\delta$, then the $C^{1}$
distance between $\tilde{\gamma}_{i}|[0,L_{i}+1]$ and $\gamma_{i}|[0,L_{i}+1]$
is less than $\epsilon$. We choose $\epsilon>0$ such that if the
$C^{1}$ distance between $\tilde{\gamma}_{i}|[0,L_{i}+1]$ and $\gamma_{i}|[0,L_{i}+1]$
is less than $\epsilon,$ $|L_{i}-\tilde{L}_{i}|<\epsilon$, and $\tilde{\gamma}_{i}(\tilde{L}_{i})=y$,
then conditions (i)-(iii) hold with $\gamma_{i}$ replaced by $\tilde{\gamma}_{i}$,
and $L_{i}$ replaced by $\tilde{L}_{i}.$

By (iv), $y$ is not $g$-conjugate to $x$ along any of $\gamma_{1},\dots,\gamma_{n}$.
We choose open neighborhoods $U_{1},\dots,U_{n}$ of $L_{1}\gamma_{1}'(0),\dots,L_{n}\gamma_{n}'(0)$
in $T_{x}M$, respectively, and an open neighborhood $U$ of $y$
in $M$, so that \[
\exp_{x,g}:U_{i}\rightarrow U\]
 is a local diffeomorphism, for $i=1,\dots,n$. By replacing $U$
and $U_{i}$ by smaller open neighhorhoods, if necessary, we may assume
that if $\tilde{\gamma}_{i}'(0)\tilde{L}_{i}\in U_{i},$ then $|L_{i}-\tilde{L}_{i}|<\epsilon$
and $|\tilde{\gamma_{i}}'(0)-\gamma_{i}'(0)|<\delta$.

If $B_{i}\subset U_{i}$ is an open ball centered at $L_{i}\gamma_{i}'(0)$
with $\overline{B_{i}}\subset U_{i}$, then $y\notin\exp_{x,g}(\partial B_{i})$
and the topological degree of $\exp_{x,g}|\partial B_{i}$ is nonzero
at $y$. Any continuous map $f_{i}:\overline{B_{i}}\to U$ that is
sufficiently close to $\exp_{x,g}|\overline{B_{i}}$ in the $C^{0}$
topology also satisfies $y\notin f_{i}(\partial B_{i})$ , and the
topological degree of $f_{i}|\partial B_{i}$ is nonzero at $y.$
This implies $y\in f_{i}(B_{i})$. (See, for instance, Theorem 1.1
of \cite{bib15}.) Now we choose a $C^{1}$-open neighborhood $\mathcal{N}_{2}\subset\mathcal{N}_{1}$
of $g$ such that if $\tilde{g}\in\mathcal{N}_{2}$, then $\exp_{x,\tilde{g}}$
is sufficiently $C^{0}$-close to $\exp_{x,g}$ on $\overline{B_{i}}$,
$i=1,\dots,n$, so that there exist $y_{i}\in B_{i}$ with $\exp_{x,\tilde{g}}y_{i}=y.$
For $\tilde{g}\in\mathcal{N}_{2}$, let $\tilde{\gamma_{i}},i=1,\dots,n$,
be $\tilde{g}$-geodesics defined on $[0,\tilde{L}_{i}]$ such that
$\tilde{\gamma}_{i}'(0)\tilde{L}_{i}=y_{i}.$ Then conditions (i)-(iii)
hold for $\gamma_{i},L_{i},g$ replaced by $\tilde{\gamma}_{i},\tilde{L}_{i},\tilde{g}$
, respectively. Thus there exists a $C^{1}$-open neighborhood $\mathcal{G}_{n}$
of $\mathcal{H}_{n}$ such that conditions (i)-(iii) hold for all
$\tilde{g}\in\mathcal{G}_{n}$.

This finishes the proof of Claim \ref{claim1}(b), and thus the proof
of Theorem \ref{thm:main}(1). The proofs of Claims \ref{claim2}(b)
and \ref{claim3}(b) are similar to the proof of Claim \ref{claim1}(b),
except we do not assume that $\tilde{\gamma_{i}}(0)=\gamma_{i}(0).$
This completes the proof of Theorem \ref{thm:main}.

\end{proof}

\end{document}